\documentclass[reqno,12pt]{article}

\usepackage{amsmath} 
\usepackage{amssymb}
\usepackage{amsthm}
\usepackage{a4wide}
\usepackage[utf8]{inputenc} 
\usepackage[english]{babel}
\usepackage{xcolor}

\usepackage[normalem]{ulem}

\numberwithin{equation}{section}

\newtheorem{theo}{Theorem}
\newtheorem*{theo*}{Theorem}
\newtheorem{prop}{Proposition}

\newtheorem{coro}{Corollary}

\newtheorem{lem}{Lemma}
\newtheorem{defi}{Definition}
\newtheorem*{defi*}{Definition}

\newtheorem{thmx}{Theorem}

\theoremstyle{remark}

\newtheorem*{Remarks*}{Remarks}

\newcommand{\bN}{\mathbb{N}}
\newcommand{\bZ}{\mathbb{Z}}
\newcommand{\bC}{\mathbb{C}}
\newcommand{\bP}{\mathbb{P}}
\newcommand{\bQ}{\mathbb{Q}}

\newcommand{\bK}{\mathbb{K}}

\newcommand{\bF}{\mathbb{F}}

\newcommand{\Qbar}{\overline{\mathbb Q}}
\newcommand{\cO}{\mathcal{O}}
\newcommand{\OK}{\mathcal{O}_{\mathbb{K}}}
\newcommand{\OKp}{\mathcal{O}_{\mathbb{K},\mathfrak{p}}}

\newcommand{\kp}{\mathfrak{p}}
\newcommand{\Gal}{\mathrm{Gal}}
\newcommand{\etilde}{\tilde{\eta}}
\newcommand{\etap}{\eta_{\kp}}

\newcommand{\cL}{\mathcal{L}}
\newcommand{\cF}{\mathcal{F}}
\newcommand{\cQ}{\mathcal{Q}}
\newcommand{\cLp}{\mathcal{L}_{\kp}}
\newcommand{\cst}{\mathrm{Cst}}
\newcommand{\Fix}{\mathrm{Fix}}

\begin{document}

\selectlanguage{english}

\title{Abel's problem, Gauss and Cartier congruences over number fields}
\date\today
\author{\'E. Delaygue and T. Rivoal}
\maketitle

\selectlanguage{english}

\begin{flushright}
{\em \`A la m\'emoire de Pierre Cartier}
\end{flushright}

\begin{abstract} Abel's problem consists in identifying the conditions under which the differential equation $y'=\eta y$, with $\eta$ an algebraic function in $\bC(x)$, possesses a non-zero algebraic solution $y$. This problem has been algorithmically solved by Risch. In a previous paper, we presented an alternative solution in the special arithmetic situation where $\eta$ has a Puiseux expansion with {\em rational} coefficients at the origin: there exists a non-trivial algebraic solution of $y'=\eta y$ if and only if the coefficients of the Puiseux expansion of $x\eta(x)$ at $0$ satisfy Gauss congruences for almost all prime numbers. In this paper, we generalize this criterion to arbitrary $\eta$ algebraic over $\Qbar(x)$, by means of a natural generalization to number fields of Gauss congruences and of the weaker Cartier congruences recently introduced in this context by Bostan. We then provide applications of this criterion in the hypergeometric setting and for Artin-Mazur zeta functions.
\end{abstract}

\section{Introduction}

Let $\eta$ be an algebraic function over $\bC(x)$. Does there exist a non-zero algebraic function solution of the following differential equation?
\begin{equation}\label{eq: Abel equation}
    y'=\eta y.
\end{equation}
This question was raised by Boulanger in \cite[p. 93]{Bou98} and referred to as Abel's problem. It naturally occurs in procedures to decide in a finite number of steps if all solutions of a linear differential equation with polynomial coefficients are algebraic over $\bC(x)$, see \cite{BD79, Sin80}. Risch \cite{Ris69} and later independently Baldassari and Dwork \cite{BD79} developed a procedure to solve Abel's problem in a finite number of steps. However, in many cases, those procedures are difficult to apply, especially when $\eta$ depends on several parameters.

In particular, motivated by questions from the theory of motives, Golyshev predicted that a non-zero algebraic solution to \eqref{eq: Abel equation} must exist whenever $x\eta(x)$ represents a generalized hypergeometric series that is both algebraic and factorial~(\footnote{We say that a generalized hypergeometric series is factorial when its Taylor coefficients at the origin can be expressed as ratios of factorials, see \S\ref{sec: Abel Hypergeometry} for a precise definition.}); see~\cite[p. 757]{Zag18}. The authors proved Golyshev's predictions in \cite{DR22} while establishing an arithmetic criterion for Equation \eqref{eq: Abel equation} to admit a non-zero algebraic solution. The latter criterion was based on Gauss congruences which turned out to be much simpler to check than to directly solve Abel's problem in the hypergeometric case. Recently, Bostan \cite{Bos24} observed that the vanishing of the $p$-curvature of \eqref{eq: Abel equation} yields another arithmetic criterion based on weaker congruences, which simplify the proof of Golyshev's predictions. Both criteria relied on the Chudnovsky's resolution of the Grothendieck--Katz conjecture in rank one \cite{CC85}, and were restricted to the case where $\eta$ admits a Puiseux expansion with rational coefficients at some point $\delta\in\bP^1(\mathbb{C})$. 
\medskip

In this paper, we generalize these criteria while proving a complete arithmetic criterion, namely Theorem \ref{theo: main} below, for Abel's problem associated to any algebraic function over $\Qbar(x)$. Naively switching from $\bQ$ to $\Qbar$ can change the outcome of Abel's problem as shown by the simple equations
$$
y'=\frac{r}{x}y\quad\textup{and}\quad y'=\frac{\sqrt{2}}{x}y,
$$
with $r\in\bQ$, where the former admits the algebraic solution $x^r$, while the non-zero solutions of the latter are all transcendental over $\mathbb{C}(x)$. 

Secondly, we apply our criterion in two different directions. We generalize Golyshev's predictions and our results in \cite{DR22}, see Theorem \ref{theo: HypAvgRoots} in \S\ref{sec: Abel Hypergeometry}, to obtain a new family of algebraic functions built from exponentials of generalized hypergeometric functions. Then, we exhibit several interesting classes of sequences satisfying the congruences involved in our main criterion. 
In particular, we briefly discuss consequences of our criterion for the algebraicity of Artin--Mazur zeta functions. 
\medskip

To state our arithmetic criterion, we first introduce some classical notions and notations of basic algebraic number theory, see \textit{e.g.} \cite{LangANT}. We consider here number fields over $\mathbb Q$ that are assumed to be Galois. We denote by $\OK$ the ring of integers of such a Galois number field $\bK$. For every rational prime number $p$, there are finitely many prime ideals $\kp$ of $\OK$ that contain $p$; we say that $\kp$ is above $p$ and from now on we will drop the word ``ideal''. If $\kp_1,\dots,\kp_n$ are the primes above $p$, then there exists a positive integer $e$ such that $p\mathcal{O}_{\mathbb K}=\kp_1^e\cdots \kp_n^e$ because $\bK/\bQ$ is Galois. If $e\geq 2$, then we say that $p$ is ramified in $\bK$, otherwise $p$ is said unramified. The set of ramified rational prime numbers of $\bK$ is finite: it is the set of the prime divisors of the absolute discriminant $\Delta_{\bK/\bQ}\in\bZ$ of the field extension $\bK/\bQ$. For every prime ideal $\kp$ we denote by $\OKp$ the localization of $\OK$ at $\kp$. For any unramified rational prime number $p$, there exists a unique automorphism $\tau_p\in\Gal(\bK/\bQ)$ such that, for every prime ideal $\kp$ above $p$, and every $x\in\OK$, we have $\tau_p(x)\equiv x^p\mod\kp\OK$, so that $\tau_p$ induces the Frobenius endomorphism of the residue field $k_\kp:=\OK/\kp\OK$.
\medskip

Gauss congruences are traditionally defined for sequences of rational numbers; see \cite{Zarelua} for a historical account. In the literature, these congruences are known under various names and are related to Dold sequences as well as realizable sequences (see~\cite{BGW21}). In this work, we generalize their definition to Galois number fields as follows.

\begin{defi}
Let $\kp$ be an unramified prime of $\OK$ above a rational prime $p$. A sequence $(a_n)_{n\in\bZ}$ of algebraic numbers in $\mathbb K$ is said to satisfy Gauss congruences for the prime $\kp$ if, for all integers $n$, we have 
$a_{np}-\tau_p(a_n)\in np\OKp$.
\end{defi}

This is equivalent to $a_{mp^{s+1}}-\tau_p(a_{mp^s})\in p^{s+1}\OKp$ for all integers $m\in\bZ$ and $s\geq 0$. We recover the usual definition when $\mathbb K=\mathbb Q$ because $\kp=(p)$ and $\tau_p=id$. 
\medskip

We now consider weaker congruences, which do not imply any congruence modulo higher powers of $p$.

\begin{defi}\label{defi: Cartier congruences}
Let $\kp$ be an unramified prime of $\OK$ above a rational prime $p$. We say that a sequence $(a_n)_{n\in\bZ}$ of algebraic numbers in $\mathbb K$ satisfies Cartier congruences for the prime $\kp$ if $
a_{np}-\tau_p(a_n)\in p\OKp$ for all integers $n$.
\end{defi}

If $\bF$ is a perfect field of characteristic $p$ endowed with the Frobenius automorphism~$\tau$, then the operator $\Lambda$ that sends an $\bF$-valued sequence $(a_n)_{n\in\bZ}$ to $(\tau^{-1}(a_{np}))_{n\in\bZ}$ is often called a Cartier operator, see \cite[p.~376]{AS03}. The analogy between Cartier operators and the congruences of Definition \ref{defi: Cartier congruences} motivated our terminology.  
\medskip

When a $\bK$-valued sequence $(a_n)_{n\in \mathbb Z}$ satisfies Gauss (resp.~Cartier) congruences for almost all primes $\kp$ (\footnote{We say that a property holds for almost all primes $\kp$ of $\OK$ when it holds for all but finitely many primes $\kp$ of $\OK$.}), we say that it has the \textit{Gauss (resp. Cartier) property}. In this case, we also say that its bilateral generating series $\sum_{n\in \mathbb Z} a_n x^n$ has the Gauss (resp. Cartier) property. In the particular situation that $a_n=0$ for all $n<r$ for some $r\in \mathbb Z$, \textit{i.e.} the generating series is a Laurent series in $\bK((x))$, we shall say that the sequence $(a_n)_{n\ge r}$ satisfies Gauss (resp. Cartier) congruences. However, it is easily shown (see Lemma \ref{lem: Cartier rational} in \S\ref{sec: proof thm main}) that in this case necessarily $a_n=0$ for negative $n$ and that $a_0\in\bQ$, or in other words that if $f\in\bK((x))$ has the Gauss or the Cartier property, then $f\in\bQ+x\bK[[x]]$. Having the Gauss or the Cartier property does not depend on the Galois number field in which the terms $a_n$, $n\in\bZ$, are embedded. 
\medskip

Our main result, proved in \S\ref{sec: proof thm main},  is the following.

\begin{theo}\label{theo: main}
Let $\eta\in\Qbar((x))$ be algebraic over $\Qbar(x)$. Then the following assertions are equivalent:
\begin{itemize}
    \item[$\mathrm{(i)}$] The differential equation $y'=\eta y$ admits a non-zero solution $y$ algebraic over $\Qbar(x)$;
    \item[$\mathrm{(ii)}$] $x\eta(x)$ has the Gauss property;
    \item[$\mathrm{(iii)}$] $x\eta(x)$ has the Cartier property.
\end{itemize}
\end{theo}
Let us first remark that if $\eta\in\Qbar((x))$ is algebraic over $\Qbar(x)$, then the Laurent coefficients of $\eta$ can be embedded into a Galois number field, so that Conditions $\mathrm{(ii)}$ and $\mathrm{(iii)}$ in Theorem \ref{theo: main} make sense. When $\Qbar$ is replaced by $\mathbb Q$ everywhere in Theorem \ref{theo: main}, the equivalence of $\mathrm{(i)}$ and $\mathrm{(ii)}$ was proved in \cite{DR22} by the authors, and the supplementary equivalence with $\mathrm{(iii)}$ was proved in \cite{Bos24} by Bostan.

Let us explain how Theorem \ref{theo: main} yields an arithmetic characterization of Abel's problem for algebraic functions over $\Qbar(x)$. Let $\varphi$ be an algebraic function over $\Qbar(x)$. Let $\delta\in\Qbar$, respectively $\delta=\infty$, and consider the Puiseux expansion of $\varphi$ at $\delta$  given respectively by 
\begin{equation}\label{eq:expansions}
\sum_{n=r}^\infty p_n (x-\delta)^{n/d}\quad\textup{and}\quad\sum_{n=r}^\infty p_n x^{-n/d},
\end{equation}
where $d\geq 1$ and $r$ are both integers. There exists a Galois number field $\bK$ such that the sequence $(p_n)_{n\in\mathbb{Z}}$ is $\bK$-valued, where $p_n=0$ for $n<r$. We say that $\varphi$ has the \textit{Gauss property} at $\delta$ if $(p_n)_{n\in\bZ}$ has the Gauss property; by the remark preceding Theorem \ref{theo: main}, the latter property implies that $p_n=0$ for $n<0$ and $p_0\in\bQ$. This is a generalization of the above definition for Laurent series in $\bK((x))$, when $d=1$ and $\delta=0$.  \medskip

A consequence of Theorem \ref{theo: main} is the following criterion, proved in \S\ref{sec: proof thm main} as well.

\begin{coro}\label{coro:super} Let $\eta$ be an algebraic function over $\Qbar(x)$ and $\delta\in\Qbar$, respectively $\delta=\infty$. Then the following assertions are equivalent:
\begin{itemize}
	\item[$\mathrm{(i)}$] The differential equation $y'=\eta y$ admits a non-zero solution $y$ algebraic over $\Qbar(x)$;
	\item[$\mathrm{(ii)}$] $(x-\delta)\eta(x)$, respectively $x\eta(x)$, has the Gauss property at $\delta$;
	\item[$\mathrm{(iii)}$] $(x-\delta)\eta(x)$, respectively $x\eta(x)$, has the Cartier property at $\delta$.
\end{itemize}
\end{coro}

It immediately follows from Corollary \ref{coro:super} that the Gauss and Cartier properties are somehow independent of the point $\delta\in\bP^1(\Qbar)$ at which we expand an algebraic function $\eta$. If $\eta$ is an algebraic function and $\delta\in\Qbar$, then we write $\eta_\delta(x):=(x-\delta)\eta(x)$ and $\eta_\infty(x)=x\eta(x)$.

\begin{coro}\label{coro:Puiseux} Let $\eta$ be an algebraic function over $\Qbar(x)$ and $\delta_1$ and $\delta_2$ two points in $\bP^1(\Qbar)$. Then $\eta_{\delta_1}$ has the Gauss (resp. Cartier) property at $\delta_1$ if and only if $\eta_{\delta_2}$ has the Gauss (resp. Cartier) property at $\delta_2$.
\end{coro}

The paper is organized as follows. In \S\ref{sec:application}, we present various applications of our results. First we give a condition ensuring that Abel's equation has a non-trivial algebraic solution when $x\eta(x)$ is a suitable $\Qbar$-linear 
combination of generalized hypergeometric series, see \S\ref{sec: Abel Hypergeometry}. Then we provide examples of sequences of Taylor coefficients having the Cartier or Gauss property (globally bounded hypergeometric series, constant terms of Laurent polynomials, number of fixed points of iterates of maps) in \S\S\ref{ssec:cartierhyp}-\ref{ssec:constantterms}-\ref{sec: Artin Mazur}. In \S\ref{sec:tools}, we present various tools (Dieudonn\'e-Dwork's lemma over number fields, $p$-curvature) needed for the proofs of Theorem~\ref{theo: main} and Corollary~\ref{coro:super}, both given in \S\ref{sec: proof thm main}.

\section{Applications} \label{sec:application}

We first generalize the results of \cite{DR22} on Abel's problem in the hypergeometric case. Then we exhibit classical classes of sequences satisfying the Cartier or Gauss  property.  In Sections 2.1 and 2.3 our arguments yield \emph{Cartier} congruences but not \emph{Gauss} congruences; strengthening them appears substantially harder with our present techniques. By contrast, in §2.4 the dynamical setting naturally provides Gauss congruences. This contrast highlights the usefulness of our Cartier-based criteria, which already suffice to deduce algebraicity in these cases.

\subsection{Abel's problem and hypergeometry}\label{sec: Abel Hypergeometry}

Our article \cite{DR22} was motivated by predictions made by Golyshev concerning the algebraicity of exponential type functions defined in terms of certain algebraic hypergeometric series. Let us remind the reader that if $\boldsymbol{\alpha}:=(\alpha_1,\dots,\alpha_r)$ and $\boldsymbol{\beta}:=(\beta_1,\dots,\beta_s)$ are tuples of rational numbers in $\mathbb{Q}\setminus\mathbb{Z}_{\leq 0}$, then the generalized hypergeometric series with rational parameters \cite{slater} is defined by
$$
{}_{r}F_s\left[\begin{array}{c}\alpha_1,\dots,\alpha_r\\ \beta_1,\dots,\beta_s \end{array} ;x\right] := \sum_{n=0}^\infty \frac{(\alpha_1)_n\cdots (\alpha_r)_n}{(\beta_1)_n\cdots (\beta_s)_n} \frac{x^n}{n!},
$$
where $(\alpha)_n$ is the Pochhammer symbol defined by $(\alpha)_n:=\alpha(\alpha+1)\cdots (\alpha+n-1)$ for $n\geq 1$ and $1$ if $n=0$. We set
$$
\mathcal{Q}_{\boldsymbol{\alpha},\boldsymbol{\beta}}(n):=\frac{(\alpha_1)_n\cdots(\alpha_r)_n}{(\beta_1)_n\cdots(\beta_s)_n}\quad\textup{and}\quad \mathcal{F}_{\boldsymbol{\alpha},\boldsymbol{\beta}}(x):=\sum_{n=0}^\infty\mathcal{Q}_{\boldsymbol{\alpha},\boldsymbol{\beta}}(n)x^n,
$$
so that
$$
\mathcal{F}_{\boldsymbol{\alpha},\boldsymbol{\beta}}(x)= {}_{r+1}F_s\left[\begin{array}{c}\alpha_1,\dots,\alpha_r,1\\ \beta_1,\dots,\beta_s \end{array} ;x\right],
$$
which has a finite positive radius of convergence if and only if $r=s$. In the rest of this section, we assume that $r=s$ because this is a necessary condition for $\mathcal{F}_{\boldsymbol{\alpha},\boldsymbol{\beta}}(x)$ to be a non-polynomial algebraic function. The series $\mathcal{F}_{\boldsymbol{\alpha},\boldsymbol{\beta}}$ is said \textit{factorial} if it can be expressed in the form 
\begin{equation} \label{eq33}
\mathcal{F}_{\boldsymbol{\alpha},\boldsymbol{\beta}}(Cx) = \sum_{n=0}^\infty \frac{(e_1 n)!(e_2 n)!\cdots (e_u n)!}{(f_1 n)!(f_2 n)!\cdots (f_v n)!} x^n,
\end{equation}
where $C$ is a non-zero rational constant, and $(e_1,\dots,e_u)$ and $(f_1,\dots,f_v)$ are tuples of positive integers. The series $\mathcal{F}_{\boldsymbol{\alpha},\boldsymbol{\beta}}$ is factorial if and only if $\boldsymbol{\alpha}$ and $\boldsymbol{\beta}$ are tuples of parameters in $\mathbb{Q}\cap(0,1]$ satisfying
$$
\frac{(x-e^{2i\pi \alpha_1})\cdots(x-e^{2i\pi \alpha_r})}{(x-e^{2i\pi \beta_1})\cdots(x-e^{2i\pi \beta_s})}\in\mathbb{Q}(x),
$$
which is equivalent to saying that $\boldsymbol{\alpha}$ and $\boldsymbol{\beta}$ are $R$-partitioned in the sense of \cite[\S 7]{delaygue1}.
\medskip

Golyshev predicted, as reported by Zagier~\cite[p.~757]{Zag18}, that the algebraicity of a factorial generalized hypergeometric series $\mathcal{F}_{\boldsymbol{\alpha},\boldsymbol{\beta}}$ implies the algebraicity of the exponential function 
$$
y_{\boldsymbol{\alpha},\boldsymbol{\beta}}:=\exp\int\frac{\mathcal{F}_{\boldsymbol{\alpha},\boldsymbol{\beta}}(x)}{x}\mathrm{d}x.
$$
We proved these predictions that turned out to be sharp. Precisely, we showed \cite[Theorem~2]{DR22} that if $\mathcal{F}_{\boldsymbol{\alpha},\boldsymbol{\beta}}$ is algebraic over $\bQ(x)$, then $y_{\boldsymbol{\alpha},\boldsymbol{\beta}}$ is algebraic if and only if $\mathcal{F}_{\boldsymbol{\alpha},\boldsymbol{\beta}}$ is factorial.
\medskip

In particular, if $\mathcal{F}_{\boldsymbol{\alpha},\boldsymbol{\beta}}$ is algebraic but non-factorial, then $y_{\boldsymbol{\alpha},\boldsymbol{\beta}}$ is transcendental over $\bQ(x)$. However, one can always consider some other algebraic hypergeometric series naturally associated with $\mathcal{F}_{\boldsymbol{\alpha},\boldsymbol{\beta}}$ such that the product of their corresponding transcendental $y_{\boldsymbol{\alpha},\boldsymbol{\beta}}$'s becomes algebraic over $\bQ(x)$, see \cite[Theorem 3]{DR22}. In this article, we further generalize this latter theorem and Golyshev's predictions as follows. Let $\{\cdot\}$ stand for the fractional part function and consider the slight modification of $\{\cdot\}$ given, for all real numbers $x$, by $\langle x\rangle=\{x\}$ if $x$ is not an integer and $1$ otherwise (instead of $0$). We define $\langle\cdot\rangle$ on tuples component-wise, that is $\langle(\alpha_1,\dots,\alpha_r)\rangle:=(\langle\alpha_1\rangle,\dots,\langle\alpha_r\rangle)$. 

\begin{theo}\label{theo: HypAvgRoots}
Let $\boldsymbol{\alpha}$ and $\boldsymbol{\beta}=(\beta_1,\dots,\beta_r)$, with $\beta_r=1$, be disjoint tuples of rational numbers in $(0,1]$ such that $\mathcal{F}_{\boldsymbol{\alpha},\boldsymbol{\beta}}$ is algebraic over $\mathbb{Q}(x)$. Let $d\ge 1$ be the least common multiple of the denominators of the $\alpha_i$'s and $\beta_j$'s written in reduced form. Let $\xi$ be a $d$-th root of unity. Then the function 
\begin{equation}\label{eq:yProduct}
\underset{\gcd(k,d)=1}{\prod_{k=1}^{d}}\exp\int \xi^k\frac{\mathcal{F}_{\langle k\boldsymbol{\alpha}\rangle,\langle k\boldsymbol{\beta}\rangle}(x)}{x}\mathrm{d}x
\end{equation}
is algebraic over $\Qbar(x)$.
\end{theo}

Setting $\xi=1$ in Theorem \ref{theo: HypAvgRoots} yields the above mentioned \cite[Theorem 3]{DR22}.

\begin{proof}
Let $\bK$ be a Galois number field containing $\xi$. For every $k\in\{1,\dots,d\}$ coprime to $d$, the tuples $\langle k\boldsymbol{\alpha}\rangle$ and $\langle k\boldsymbol{\beta}\rangle$ are disjoint. By the criterion of Beukers and Heckmann \cite[Theorem 4.8]{BH89}, see \cite[Section 4.2]{DR22}, the series $\cF_{\langle k\boldsymbol{\alpha}\rangle,\langle k\boldsymbol{\beta}\rangle}$ is algebraic over $\bQ(x)$. Write
$$
\eta(x):=\frac{1}{x}\underset{\gcd(k,d)=1}{\sum_{k=1}^d} \xi^k\mathcal{F}_{\langle k\boldsymbol{\alpha}\rangle,\langle k\boldsymbol{\beta}\rangle}(x)\quad\textup{and}\quad \mathcal{Q}(n):=\underset{\gcd(k,d)=1}{\sum_{k=1}^d}\xi^k\mathcal{Q}_{\langle k\boldsymbol{\alpha}\rangle,\langle k\boldsymbol{\beta}\rangle}(n).
$$
Then $\eta$ is algebraic over $\bK(x)$ and the series in \eqref{eq:yProduct} is solution of $y'=\eta y$. By Theorem~\ref{theo: main}, it remains to prove that the sequence $(\cQ(n))_{n\geq 0}$ of the Taylor coefficients at $0$ of $x\eta(x)$ has the Cartier property.
\medskip

For every $k\in\{1,\dots,d\}$ coprime to $d$, the series $\mathcal{F}_{\langle k\boldsymbol{\alpha}\rangle,\langle k\boldsymbol{\beta}\rangle}$ is algebraic over $\bQ(x)$ so that, for every prime $p$, we can apply \cite[Lemma 2]{DR22} with $\langle k\boldsymbol{\alpha}\rangle$ and $\langle k\boldsymbol{\beta}\rangle$ instead of $\boldsymbol{\alpha}$ and $\boldsymbol{\beta}$ respectively. We obtain that, for every sufficiently large prime $p$ and all non-negative integers $n$, we have
$$
\cQ_{\langle k\boldsymbol{\alpha}\rangle,\langle k\boldsymbol{\beta}\rangle} (np)-\cQ_{\langle a\langle k\boldsymbol{\alpha}\rangle\rangle,\langle a\langle k\boldsymbol{\beta}\rangle\rangle}(n)\in p\bZ_{(p)},
$$
where $a\in\{1,\dots,d\}$ satisfies $ap\equiv 1\mod d$. Since $\langle a\langle k\boldsymbol{\alpha}\rangle\rangle=\langle ak\boldsymbol{\alpha}\rangle$ and $\langle a\langle k\boldsymbol{\beta}\rangle\rangle=\langle ak\boldsymbol{\beta}\rangle$, we have
\begin{equation}\label{eq: Qcong}
    \cQ_{\langle k\boldsymbol{\alpha}\rangle,\langle k\boldsymbol{\beta}\rangle} (np)-\cQ_{\langle a k\boldsymbol{\alpha}\rangle,\langle ak\boldsymbol{\beta}\rangle}(n)\in p\bZ_{(p)}.
\end{equation}

For almost all primes $\kp$ of $\OK$, all $k\in\{1,\dots,d\}$ coprime to $d$ and all non-negative integers $n$, $\mathcal{F}_{\langle k\boldsymbol{\alpha}\rangle,\langle k\boldsymbol{\beta}\rangle}$ is algebraic hence $\xi^k$ and $\cQ_{\langle k\boldsymbol{\alpha}\rangle,\langle k\boldsymbol{\beta}\rangle}(n)$ both belong to $\OKp$. Furthermore, given $\kp$ above the rational prime $p$ with $\tau_p$ the Frobenius element associated with $p$, and $a\in\{1,\dots,d\}$ satisfying $ap\equiv 1\mod d$, we have the congruence
\begin{equation}\label{eq: xicong}
\tau_p(\xi^{ak})\equiv \xi^{akp}\equiv \xi^k\mod \kp\OKp,
\end{equation}
since $\xi^k$ is a $d$-th root of unity. Then, for every non-negative integer $n$, Eqs.~\eqref{eq: xicong} and \eqref{eq: Qcong} yield
\begin{align*}
\cQ(np)&=\underset{\gcd(k,d)=1}{\sum_{k=1}^d}\xi^k\mathcal{Q}_{\langle k\boldsymbol{\alpha}\rangle,\langle k\boldsymbol{\beta}\rangle}(np)\\
&\equiv \underset{\gcd(k,d)=1}{\sum_{k=1}^d}\tau_p(\xi^{ak})\mathcal{Q}_{\langle ak\boldsymbol{\alpha}\rangle,\langle ak\boldsymbol{\beta}\rangle}(n)\mod \kp\OKp\\
&\equiv \tau_p(\cQ(n))\mod \kp\OKp.
\end{align*}
This shows that $x\eta(x)$ has the Cartier property and ends the proof of Theorem \ref{theo: HypAvgRoots}.
\end{proof}

As an illustration, let $\boldsymbol{\alpha}$ and $\boldsymbol{\beta}$ be such that
$$
\mathcal{F}_{\boldsymbol{\alpha},\boldsymbol{\beta}}(x)=\sum_{n=0}^\infty\frac{(1/4)_n(11/12)_n}{(1/2)_nn!}x^n.
$$
As explained in \cite{DR22}, the interlacing condition of Beukers and Heckmann yields four algebraic hypergeometric series: $f_1,f_5,f_7$ and $f_{11}$ which are respectively defined by their Taylor coefficients at $0$: 
$$
\frac{(1/4)_n(11/12)_n}{(1/2)_nn!},\quad \frac{(1/4)_n(7/12)_n}{(1/2)_nn!},\quad \frac{(5/12)_n(3/4)_n}{(1/2)_nn!}\quad\textup{and}\quad \frac{(1/12)_n(3/4)_n}{(1/2)_nn!}.
$$
Since none of those series is factorial, \cite[Theorem 2]{DR22} implies that, for all $i$ in $\{1,5,7,11\}$, the function
$$
y_i(x)=\exp\int\frac{f_i(x)}{x}\mathrm{d}x
$$
is transcendental over $\mathbb{Q}(x)$. But Theorem \ref{theo: HypAvgRoots} shows that, for any $12$-th root of unity $\xi$, the product 
$$
y_1^\xi y_5^{\xi^5}y_7^{\xi^7}y_{11}^{\xi^{11}}
$$ 
is algebraic over $\Qbar(x)$.

\medskip

In the rest of this section, we consider the following natural question (in view of Theorem~\ref{theo: main}): what are interesting examples of sequences with the Cartier or the Gauss property? We exhibit classical classes of sequences with one or the other, or both, of these properties.

\subsection{Cartier and Gauss properties and hypergeometry} \label{ssec:cartierhyp}

If $\cF_{\boldsymbol{\alpha},\boldsymbol{\beta}}$ is an algebraic hypergeometric series, then by Theorem \ref{theo: main} and \cite[Theorem 2]{DR22}, we obtain that $\cF_{\boldsymbol{\alpha},\boldsymbol{\beta}}$ has the Cartier (resp. Gauss) property if and only if $\cF_{\boldsymbol{\alpha},\boldsymbol{\beta}}$ is factorial. But we want to emphasis the fact that this equivalence remains true in a more general setting.

Let us recall that a power series $f\in\Qbar[[x]]$ is said globally bounded if it has a positive radius of convergence and if there is a non-zero integer $C$ such that $Cf(Cz)$ has algebraic integers coefficients. By Eisenstein's theorem, every algebraic power series $f\in\Qbar[[x]]$ is globally bounded. Then the proof of \cite[Theorem 4]{DR22} yields the following criterion.

\begin{prop}
	Let $\cF_{\boldsymbol{\alpha},\boldsymbol{\beta}}$ be a globally bounded hypergeometric series. Then the following assertions are equivalent.
	\begin{itemize}
	\item[$\mathrm{(i)}$]$\cF_{\boldsymbol{\alpha},\boldsymbol{\beta}}$ is factorial;
	\item[$\mathrm{(ii)}$] $\cF_{\boldsymbol{\alpha},\boldsymbol{\beta}}$ has the Gauss property;
	\item[$\mathrm{(iii)}$] $\cF_{\boldsymbol{\alpha},\boldsymbol{\beta}}$ has the Cartier property.	
	\end{itemize}
\end{prop}

This completely characterizes hypergeometric series with the Gauss (resp. Cartier) property.

\subsection{Constant terms of powers of Laurent polynomials}\label{ssec:constantterms}

A classical result of Furstenberg \cite{Furstenberg} says that algebraic functions are diagonals of bivariate rational fractions. In some particular cases, one can prove that such diagonals have the Gauss property; see the results in \cite{BHS18} and \cite{BV21}, as well as \cite[Section 2.1]{DR22} for more details. But there also exists a large related class of sequences that has the Gauss property: the generating series of constant terms of powers of Laurent polynomials.
\medskip

Let $\bK$ be a Galois number field, $r$ a positive integer and $\lambda\in\bK[x_1^\pm,\dots,x_r^\pm]$ a Laurent polynomial with coefficients in $\bK$. We denote by $\cst(\lambda)$ the constant term of $\lambda$. We set
$$
f_\lambda(x):=\sum_{n=0}^\infty \cst(\lambda^n) x^n.
$$
When the Newton polygon of $\lambda$ has only one internal integral point, then $f_\lambda$ satisfies strong congruences modulo $\kp$ for almost all primes $\kp$ of $\OK$. Precisely, when $\bK=\bQ$, Samol and van Straten \cite{SvS15} proved that such $f_\lambda$'s satisfy Lucas congruences, while Mellit and Vlasenko \cite{MV16} proved that $f_\lambda$ in fact satisfies the more restrictive Dwork congruences.

Without any assumption on the Newton polygon of $\lambda$, Bostan, Straub, and Yurkevich have shown in \cite{BSY23} that $f_\lambda$ satisfies the Gauss property when $\bK = \bQ$. Their proof can be extended without difficulty to the case where $\bK$ is a number field.
\medskip

Another classical example of generating series with the Gauss property is given by 
$$
f_A(x):=\sum_{n=0}^\infty\mathrm{Tr}(A^n)x^n,
$$
where $A$ is any square matrix of algebraic integers. Gauss congruences for $f_A$ were first established by J\"anichen \cite{Janichen} in 1921  when the coefficients of $A$ are rational integers. We refer the reader to \cite{Steinlein} for a historical overview of Gauss congruences for the traces of powers of matrices.

\medskip

A way to study both series $f_\lambda$ and $f_A$ simultaneously is to consider the following construction. Given a positive integer $r$ and a square matrix $A$ with entries in $\bK[x_1^\pm,\dots,x_r^\pm]$, we set
$$
g_A(x):=\sum_{n=0}^\infty \mathrm{cst}(\mathrm{Tr}(A^n))x^n.
$$

\begin{prop}\label{prop: Laurent Cartier}
	Let $r\ge 1$ be an integer and $A$ a square matrix with entries in $\bK[x_1^\pm,\dots,x_r^\pm]$. Then the sequence $\mathrm{cst}(\mathrm{Tr}(A^n))$ has the Cartier property.
\end{prop}

\begin{proof}
For almost every prime $\kp$ of $\OK$, all entries of $A$ lie in $\OKp[x_1^\pm, \dots, x_r^\pm]$. Let us now fix such a prime $\kp$ that is additionally unramified in $\bK$. We denote by $A_\kp$ the reduction of $A$ modulo $\kp$, obtained while reducing modulo $\kp$ each coefficient of each entry of $A$. Hence $A_\kp$ is a square matrix with entries in $k_\kp[x_1^\pm,\dots,x_r^\pm]$, where $k_\kp$ denotes the residue field of $\bK$ modulo $\kp$. 

We denote by $\tau$ the Frobenius element of $\bK$ associated with $\kp$. This automorphism naturally extends to a Frobenius lift in $\bK[x_1^\pm,\dots,x_r^\pm]$ that we still denote $\tau$. It corresponds to the endomorphism sending $\alpha$ to $\tau(\alpha)$ for $\alpha\in\bK$ and $x_i$ to $x_i^p$ for $i\in\{1,\dots,r\}$. For every Laurent polynomial $\lambda\in\OKp[x_1^\pm,\dots,x_r^\pm]$, we have $\lambda^p\equiv\tau(\lambda)\mod \kp$.
\medskip

Let $\Omega_\kp$ denotes an algebraic closure of $k_\kp[x_1^\pm,\dots,x_r^\pm]$. Let $\lambda_1,\dots,\lambda_s$ denote the eigenvalues of $A_\kp$ in $\Omega_\kp$, counted with multiplicities. Let $p$ be the characteristic of $k_\kp$, then, for every $n\in\bN$, we have
\begin{align*}
\mathrm{cst}(\mathrm{Tr}(A^{np}))&= \mathrm{cst}(\lambda_1^{np}+\cdots+\lambda_s^{np})\\
&= \mathrm{cst}((\lambda_1^n+\cdots +\lambda_s^n)^p)\\
&= \mathrm{cst}(\mathrm{Tr}(A^n)^p)\\
&\equiv \mathrm{cst}(\tau(\mathrm{Tr}(A^n)))\mod \kp\\
&\equiv\tau(\mathrm{cst}(\mathrm{Tr}(A^n)))\mod \kp,
\end{align*}
as desired.
\end{proof}

As a consequence of Theorem \ref{theo: main} and Proposition \ref{prop: Laurent Cartier}, we obtain the following result.

\begin{coro}
For every square matrix $A$ with entries in $\bK[x_1^\pm,\dots,x_r^\pm]$, then 
\begin{equation}\label{eq: zeta A}
g_A\quad\textup{and}\quad\exp\int\frac{g_A(x)}{x}\mathrm{d}x
\end{equation}
are either both algebraic or both transcendental over $\Qbar(x)$.
\end{coro}

If the variables $x_1, \dots, x_r$ do not commute and the Laurent polynomial ring $\bK[x_1^\pm, \dots, x_r^\pm]$ is replaced by the free group algebra $\bQ\langle x_1, x_1^{-1}, \dots, x_r, x_r^{-1} \rangle$, Kassel and Reutenauer showed in \cite{KR14} that the series in \eqref{eq: zeta A} are always algebraic, thereby generalizing Kontsevich's result \cite{Kon11} in the case where $A$ is a scalar matrix.

\subsection{Artin--Mazur zeta functions}\label{sec: Artin Mazur}

Finally, it is worth noticing that a simple construction from the theory of dynamical systems gives rise to a rich class of sequences with the Gauss (and hence the Cartier) property.
\medskip

Let $X$ be a set and $f:X\rightarrow X$ a map on $X$. If, for every positive integer $n$, the number $|\Fix(f^n)|$ of fixed points of the $n$-th iterate of $f$ is finite, then one can consider the Artin--Mazur zeta function associated with $f$:
$$
Z_f(x):=\exp\left(\sum_{n=1}^\infty\frac{\vert \Fix(f^n)\vert }{n}x^n\right)\in\bQ[[x]].
$$ 

For example, if $X$ is an affine variety over $\bF_p$ and $f$ is the Frobenius map defined on $X(\overline{\bF}_p)$, then, for every positive integer $n$, the number of points of $X$ in $\bF_{p^n}$ is 
$\vert\Fix(f^n)\vert$, so the local zeta function of $X$ is $Z_f$. In this case, it is known since \cite{Dwo60} that $Z_f$ is a rational function with integer Taylor coefficients at the origin. We refer the reader to \cite{BGNS23} for an overview of questions concerning the transcendental, algebraic, or rational nature of various Artin--Mazur zeta functions.

The following result is well-known, but we include its proof for the sake of completeness.

\begin{prop}\label{prop: Fix Gauss}
Let $X$ be a set and $f:X\rightarrow X$ a map such that, for every positive integer $n$, the set $\Fix(f^n)$ is finite. Then the sequence $(\vert \Fix(f^n)\vert)_{n\geq 1}$ satisfies Gauss congruences for every prime $p$.
\end{prop}

We obtain the following result as a consequence of Proposition \ref{prop: Fix Gauss} and Theorem \ref{theo: main}.

\begin{coro}\label{coro: Z_f}
	Let $X$ be a set and $f:X\rightarrow X$ a map such that, for every positive integer $n$, the set $\Fix(f^n)$ is finite. Then $Z_f\in\bZ[[x]]$ and
	$$
	Z_f(x)\quad\textup{and}\quad \sum_{n=1}^\infty \vert \Fix(f^n)\vert x^n
	$$
	are both algebraic or both transcendental over $\Qbar(x)$.
\end{coro}

The integrality of the coefficients of $Z_f$ in Corollary \ref{coro: Z_f} follows directly from Gauss congruences for any prime, using the Dieudonn\'e-Dwork lemma; see Corollary \ref{coro:1} below.

\begin{proof}[Proof of Proposition \ref{prop: Fix Gauss}]
For every positive integer $n$, let $\cO_n$ denote the number of orbits of $f$ with exact size $n$. Then a point of $X$ is fixed by $f^n$ if and only if the size of its orbit divides $n$. That is
$$
\vert \Fix(f^n)\vert=\sum_{d\mid n}d\cO_d.
$$
By M\"{o}bius inversion formula, we obtain that
\begin{equation}\label{eq: Necklace}
\sum_{d\mid n}\mu\left(\frac{n}{d}\right)|\Fix(f^d)|=n\cO_n\equiv 0\mod n.
\end{equation}
Let $p$ be a fixed prime number and $m$ a positive integer coprime to $p$. Then \eqref{eq: Necklace} yields
\begin{align*}
\sum_{d\mid mp^{s+1}}\mu\left(\frac{mp^{s+1}}{d}\right)\vert \Fix(f^d)\vert&= \sum_{d\mid m}\left(\mu\left(\frac{m}{d}\right)\vert\Fix(f^{dp^{s+1}})\vert+ \mu\left(\frac{mp}{d}\right)\vert\Fix(f^{dp^s})\vert\right)\\
&=\sum_{d\mid m}\mu\left(\frac{m}{d}\right)(\vert\Fix(f^{dp^{s+1}})\vert-\vert\Fix(f^{dp^s})\vert)\\
&\equiv 0\mod p^{s+1}.	
\end{align*}
By induction on $m$, one proves that, for every positive integer $m$, we have
$$
\vert\Fix(f^{mp^{s+1}})\vert-\vert\Fix(f^{mp^s})\vert\equiv 0\mod p^{s+1},
$$
as desired. \end{proof}

\section{Tools in $p$-adic analysis}\label{sec:tools}

In this section, we gather and prove various results in $p$-adic analysis needed for the proof of our results.

\subsection{On the $p$-curvature of equations of order $1$} \label{ssec:pcurvature}

We refer the reader to the survey \cite{BCR24} for a complete introduction to the $p$-curvature of linear differential equations over $\bQ(x)$. For the sake of completeness, we introduce in this section the notion of $p$-curvature of differential equations of order $1$ over algebraic functions in $\bK((x))$, where $\bK$ is a number field. Let $\eta\in\bK((x))$ be an algebraic Laurent series over $\bK(x)$. We consider the differential equation
$$
\cL:= \partial -\eta,
$$
where $\partial$ is the usual derivation of $\bK((x))$. By Eisenstein's theorem, for almost all primes $\kp$ of $\OK$, we have $\eta\in\OKp((x))$ and we denote by $\eta_\kp\in k_\kp((x))$ the reduction of $\eta$ modulo $\kp$, where $k_\kp$ stands for the residue field of $\bK$ at $\kp$. For every such prime $\kp$, we consider the reduction of $\cL$ modulo $\kp$:
$$
\cLp=\partial-\etap,
$$
where we still denote by $\partial$ the usual derivation of $k_\kp((x))$. Let us remind the reader that $k_\kp$ is a finite field of characteristic $p$. We consider the map
$$
\begin{array}{rcl}
\Delta: k_\kp((x)) & \rightarrow & k_\kp((x))\\
f & \mapsto & f'-\etap f
\end{array}.
$$
The map $\Delta$ satisfies the following Leibniz rule: for every $f$ and $g$ in $k_\kp((x))$, we have
\begin{equation}\label{eq: Leibniz rule}
\Delta(fg)=f'g+f\Delta(g).
\end{equation}
In particular, $\Delta$ is a $k_\kp((x^p))$-linear map.

\begin{defi}
The $p$-th iterate of $\Delta$, $\Delta^p:k_\kp((x)) \rightarrow k_\kp((x))$, is called the $p$-curvature of~$\cL$.
\end{defi}
Obviously, $\Delta^p$ is $k_\kp((x^p))$-linear, but a key feature of the $p$-curvature is that it is in fact a $k_\kp((x))$-linear map.

\begin{prop}\label{prop: p-curvature linearity}
The $p$-curvature of $\cL$ is a $k_\kp((x))$-linear map.
\end{prop}

It follows in particular that $\Delta^p$ is completely determined by the value $\Delta^p(1)$.

\begin{proof}
By the Leibniz rule \eqref{eq: Leibniz rule}, one can show that, for every integer $n\geq 0$ and all $f,g\in k_\kp((x))$, we have
$$
\Delta^n(fg)=\sum_{k=0}^n\binom{n}{k}f^{(k)}\Delta^{n-k}(g).
$$
Taking $n=p$, we obtain that $\Delta^p(fg)=f\Delta^p(g)+f^{(p)}g=f\Delta^p(g)$, as desired.
\end{proof}

Proposition \ref{prop: p-curvature linearity} yields the following criterion. 

\begin{coro} \label{coro:2}
The $p$-curvature of $\cL$ is zero if and only if $\cLp$ has a non-zero power series solution in $k_\kp[[x]]$. 
\end{coro}

\begin{proof}
If $\cLp$ has a non-zero solution $f\in k_\kp[[x]]$, then $\Delta^p(f)=0$ and, by Proposition~\ref{prop: p-curvature linearity}, $\Delta^p$ vanishes on all of $k_\kp((x))$. Reciprocally, assume that $\Delta^p$ is identically zero, then the kernel of $\Delta$ is non-zero and there is a non-zero solution $f\in k_\kp((x))$ of $\cLp$. Thus for a suitable integer $k\ge 0$, we deduce from the $ k_\kp((x))$-linearity of $\Delta$ that $x^{kp} f(x)\in k_\kp[[x]]$ is a non-zero solution of $\cLp$, as desired.
\end{proof}

The proof of our criterion Theorem \ref{theo: main} relies on the resolution of the Grothendieck--Katz conjecture for order $1$ equations by the Chudnovsky's \cite[Theorem~8.1]{CC85}. 

\begin{thmx}[Chudnovsky--Chudnovsky \cite{CC85}]\label{theo: CC rank one}
Let $\eta\in\bK((x))$ be algebraic over $\bK(x)$. Then the differential equation $y'=\eta y$ admits a non-zero algebraic solution if and only if its $p$-curvature vanishes for almost all primes $\kp$ of $\OK$.
\end{thmx}

We shall also need a key arithmetic property of differential equations of order $1$, {\em i.e.} we have a simple formula for the $p$-curvature of $\cL$. It is due to Jacobson \cite{Jac37}; see also \cite[Theorem~3.12]{BCR24}).

\begin{thmx}\label{theo: Jacobson}
Let $\kp$ be a prime above a rational prime $p$ such that $\eta\in\OKp((x))$. Then we have $\Delta^p(1)=-\eta^{(p-1)}+(-\eta)^p$ modulo $\kp\OKp((x))$.
\end{thmx}

\begin{proof}
We adapt the proof given in \cite{BCR24} to our case. For every nonnegative integer $k$, we denote by $b_k\in k_\kp((x))$ the constant term of the differential operator $(\partial-\eta)^k$, so that $\Delta^p(1)=b_p$. The proof of Jacobson's formula in \cite[Theorem 3.12]{BCR24}) is given for $\eta\in\bF_p(x)$ but it extends \textit{mutatis mutandis} to our case $\eta\in k_\kp((x))$. This yields $b_p=-\eta^{(p-1)}+(-\eta)^p$, as desired.
\end{proof}

\subsection{The Dieudonn\'e--Dwork lemma over number fields}

Here and in the sequel, we adopt the following notation. If $\bK$ is a number field, $\tau$ an endomorphism of $\bK$ and $F$ is the generating series of a $\bK$-valued sequence $(a_n)_{n\geq 0}$, then we set
$$
F^\tau(x):=\sum_{n=0}^\infty\tau(a_n)x^n.
$$

A fundamental tool for studying arithmetic properties of exponentials is the Dieudonn\'e--Dwork lemma \cite[Lemma 1]{Dwo58}, see also \cite[p. 53]{DGS94}. It is usually stated in the framework of $p$-adic analysis and holds true in the completion of the maximal unramified extension of $\bQ_p$ in $\mathbb{C}_p$, where $\bC_p$ stands for the completion of the algebraic closure of $\bQ_p$. This lemma implies the following one, whose formulation over number fields better suits our purpose.

\begin{lem}[Dieudonn\'e--Dwork]\label{lem: DD} Let $\bK$ be a number field, $\kp$ an unramified prime of $\bK$ above a rational prime $p$, and $\tau$ the Frobenius endomorphism associated with $\kp$. Let $F\in 1+x\mathbb \bK[[x]]$. Then $F(x)\in 1+x\mathbb \OKp[[x]]$ if and only if $F^\tau(x^p)/F(x)^p \in 1+px\mathbb \OKp[[x]]$.
\end{lem}

In our context, we aim at studying $\exp(s(x))$ where $s(x)\in x\bK[[x]]$. This lemma yields the following important result that relates integrality properties of $\exp(s(x))$ to congruences satisfied by $s(x)$, which are often easier to study.

\begin{coro}\label{coro:1} Let $\bK$ be a number field, $\kp$ an unramified prime of $\bK$ above a rational prime $p$, and $\tau$ the Frobenius endomorphism associated with $\kp$. Let $s\in x\mathbb \bK[[x]]$. Then $\exp(s(x))\in 1+x\OKp[[x]]$ if and only if $s^\tau(x^p)-ps(x)\in px\OKp[[x]]$.
\end{coro}

\begin{proof}
The proof is based on that of \cite[Corollary 6.7]{LY98} that we extend to number fields for the sake of completeness. Set $F(x)=\exp(s(x))\in 1+ x\bK[[x]]$.

\textit{Proof of the if part}. Assume that $s^\tau(x^p)-ps(x)=pt(x)$ for some $t\in x\OKp[[x]]$. By Legendre's formula, for every positive integer $n$, we have
$$
v_p(n!)=\sum_{\ell=1}^\infty\left\lfloor\frac{n}{p^\ell}\right\rfloor < \sum_{\ell=1}^\infty\frac{n}{p^\ell}\leq \frac{n}{p-1},
$$
so that $v_p(n!) < n$. It follows that
$$
\frac{F^\tau(x^p)}{F(x)^p}=e^{pt(x)}=1+\sum_{n=1}^\infty \frac{p^n}{n!}t(x)^n\in 1+px\OKp[[x]].
$$
By Lemma \ref{lem: DD}, we obtain that $F(x)\in 1+x\OKp[[x]]$ as desired.

\textit{Proof of the only if part}. Assume that  $F\in 1+x\OKp[[x]]$. By Lemma \ref{lem: DD}, there exists $t\in x\OKp[[x]]$ such that
$$
\exp(s^\tau(x^p)-ps(x))=\frac{F^\tau(x^p)}{F(x)^p}=1+pt(x),
$$
so that
$$
s^\tau(x^p)-ps(x)=\log(1+pt(x))=\sum_{n=1}^\infty (-1)^{n+1}\frac{p^nt(x)^n}{n}\in px\OKp[[x]],
$$
as desired.
\end{proof}

\section{Proofs of the arithmetic criterion and of its consequences}\label{sec: proof thm main}

This section is devoted to the proofs of Theorem \ref{theo: main} and of Corollary \ref{coro:super}. We first need the following result on sequences satisfying the Cartier property.

\begin{lem}\label{lem: Cartier rational}
Let $\bK$ be a Galois number field and $\eta\in\bK((x))$ be a Laurent series satisfying the Cartier property (at $0$). Then $\eta\in\bQ+x\bK[[x]]$.
\end{lem}

\begin{proof}
Write
$$
\eta(x)=\sum_{n=r}^\infty a_n x^n\in\bK((x)),
$$
where $r\in\bZ$. Then, for almost all primes $\kp$, and all $n\in\bZ$, $\kp$ is unramified in $\bK$ and we have
$$
a_{np}-\tau_p(a_n)\in p\OKp,
$$
where $\tau_p\in\Gal(\bK/\bQ)$ is the Frobenius element associated with $\kp$, and $\kp$ is above the rational prime $p$. In particular, when $p>|r|$ and $n$ is a negative integer, we obtain that $a_{np}=0$ so $\tau_p(a_n)\in p\OKp$. Since the Galois group $\Gal(\bK/\bQ)$ is finite, there exists $\sigma\in\Gal(\bK/\bQ)$ such that $\sigma(a_n)\in p\OKp$ for infinitely many primes $\kp$. Hence $\sigma(a_n)=0$ and $a_n=0$. This proves that $\eta\in\bK[[x]]$.
\medskip

In addition, for almost all primes $\kp$ we have $a_0-\tau_p(a_0)\in p\OKp$. By Chebotarev's density theorem \cite[Ch. VIII, Theorem 10, p. 169]{LangANT}, for every $\sigma\in\Gal(\bK/\bQ)$, there exist infinitely many primes $\kp$ such that $\tau_p=\sigma$. For every such $\kp$, we obtain that $a_0-\sigma(a_0)\in p\OKp$, hence $a_0=\sigma(a_0)$. Since $\sigma$ is arbitrary in $\Gal(\bK/\bQ)$, it follows that $a_0\in\bQ$, as desired.
\end{proof}

We fix $\eta\in\Qbar((x))$ algebraic over $\Qbar(x)$. We also fix a Galois number field $\bK$ such that $\eta\in\bK((x))$. Observe that we trivially have $\mathrm{(ii)} \Rightarrow\mathrm{(iii)}$ in Theorem \ref{theo: main}. So it remains to prove the implications $\mathrm{(i)} \Rightarrow\mathrm{(ii)}$ and $\mathrm{(iii)} \Rightarrow\mathrm{(i)}$.

\subsection{Proof of $\mathrm{(i)}\Rightarrow\mathrm{(ii)}$ in Theorem \ref{theo: main}}

Assume that there exists a non-zero algebraic function $y$ such that $y'=\eta y$. Let us prove that $x\eta(x)$ has the Gauss property at $0$. By the Newton-Puiseux theorem \cite[p. 
68, Proposition~8]{serre}, $y$ admits a Puiseux expansion:
$$
y(x)=\sum_{n=r}^\infty b_nx^{n/d}=x^{r/d}\sum_{n=0}^\infty b_{n+r}x^{n/d}\in\mathbb{C}((x^{1/d})),
$$
for some integers $r$ and $d\geq 1$ and with $b_r$ non-zero. We also have
$$
y'(x)=\frac{1}{d}\sum_{n=r}^\infty nb_nx^{n/d-1}=\frac{x^{r/d-1}}{d}\sum_{n=0}^\infty (n+r)b_{n+r}x^{n/d}.
$$
It follows that $\eta=y'/y$ admits a Puiseux expansion of the form
\begin{equation}\label{Puiseuxeta2}
\eta(x)=\sum_{n=0}^\infty c_nx^{n/d-1}\in\mathbb{C}((x^{1/d})),
\end{equation}
with $c_0=r/d$. By uniqueness of the Puiseux expansion of $\eta$ at $0$ and since by assumption $\eta\in\bK((x))$, we obtain that $x\eta(x)\in\bK[[x]]$. Write
$$
x\eta(x)=\sum_{n=0}^\infty a_n x^n\quad\textup{and}\quad y(x)=x^{a_0}\exp\bigg(\sum_{n=1}^\infty \frac{a_{n}}n x^{n}\bigg)=x^{a_0}g(x), 
$$
where $g(x)\in\bK[[x]]$ and $a_0=r/d\in\bQ$. 

It follows that $g(x)=x^{-a_0}y(x)\in 1+x\bK[[x]]$ is also algebraic over $\Qbar(x)$, and there exists a positive integer $\lambda$ such that $g(\lambda x)\in\OK[[x]]$ by Eisenstein's theorem. Hence, $g\in 1+x\mathbb \OKp[[x]]$ for any prime $\kp$ of $\OK$ that does not divide $\lambda$. By Corollary~\ref{coro:1}, for any prime number $p$ that does not divide $\lambda$, and any prime $\kp$ above $p$, we have
$$
\sum_{n=1}^\infty\frac{\tau(a_n)}{n}x^{np}-p\sum_{n=1}^\infty\frac{a_n}{n}x^n\in px\OKp[[x]].
$$
In particular, for every $n\geq 1$, we have
$$
\frac{\tau(a_n)}{n}-p\frac{a_{np}}{np}\in p\OKp,
$$
which yields $\tau(a_n)- a_{np}\in np\OK$. Since $a_0\in\bQ$, this latter congruence also holds for $n=0$. It follows that $x\eta(x)$ satisfies Gauss congruences for the prime $\kp$. Since almost all primes $\kp$ do not divide $\lambda$, $x\eta(x)$ has the Gauss property.$\hfill\square$ 

\subsection{Proof of $\mathrm{(iii)}\Rightarrow\mathrm{(i)}$ in Theorem \ref{theo: main}}

We assume that $x\eta(x)$ has the Cartier property. By Lemma \ref{lem: Cartier rational}, we have
$$
x\eta(x)=\sum_{n=0}^\infty a_n x^n\in\bQ +x\bK[[x]].
$$
By Eisenstein's theorem, there exists a positive integer $\lambda$ such that, for every $n\geq 1$, we have $\lambda^n a_n\in\OK$, so that
$$
\etilde(x):=\lambda\eta(\lambda x)-\frac{a_0}{x}\in\OK[[x]].
$$
Let us show that $x\etilde(x)$ has the Cartier property. If $p$ is a prime number not dividing $\lambda$ and $\kp$ is a prime of $\OK$ above $p$, then for every positive integer $n$, we have
\begin{align*}
    \lambda^{np}a_{np}-\tau_p(\lambda^na_n)&=\lambda^{np}a_{np}-\lambda^n\tau_p(a_n)\\
    &=\lambda^{np}(a_{np}-\tau_p(a_n))+(\lambda^{np}-\lambda^n)\tau_p(a_n)\\
    &\equiv \lambda^{np}(a_{np}-\tau_p(a_n))\mod p\OKp,
\end{align*}
because $\tau_p(a_n)\in\OKp$ as $\lambda^n a_n\in\OK$ and $\tau_p$ sends $\OKp$ to $\OKp$. Since $x\eta(x)$ has the Cartier property and $\lambda$ is an integer, we obtain that $\lambda^{np}(a_{np}-\tau_p(a_n))\in p\OKp$. It follows that $\lambda^{np}a_{np}- \tau_p(\lambda^na_n)\in p\OKp$, and $x\etilde(x)$ has the Cartier property, as expected. We write
$$
\etilde(x)=\sum_{n=1}^\infty b_n x^{n-1}\in \OK[[x]].
$$

We are now in position to extend Bostan's argument \cite{Bos24} to the Galois number field case. We shall prove that for almost all prime numbers $p$ and every prime $\kp$ above $p$, the $p$-curvature of $y'=\etilde y \mod\kp$ vanishes. Let $p$ be an odd prime number and $\kp$ a prime above $p$. By Theorem \ref{theo: Jacobson} in \S\ref{ssec:pcurvature}, the $p$-curvature of $y'=\etilde y \mod\kp$ is 
equal to $-\etilde^{(p-1)}-\etilde^p\mod\kp$. We have
\begin{align*}
\etilde^{(p-1)}(x)&=\sum_{n=1}^\infty b_n(n-1)\cdots (n-p+1)x^{n-p}\\
&\equiv -\sum_{n=1}^\infty b_{np}x^{(n-1)p}\mod\kp\OK[[x]],
\end{align*}
because $b_n\in\OK$ and $(n-1)\cdots(n-p+1)$ is divisible by $p$ if $n$ is coprime to $p$, while $(n-1)\cdots (n-p+1)\equiv -1\mod p$ by Wilson's theorem when $p$ divides $n$. Furthermore, since $\etilde\in\OK[[x]]$, we have 
$$
\etilde(x)^p\equiv \sum_{n=1}^\infty \tau_p(b_n) x^{(n-1)p}\mod \kp\OK[[x]].
$$
Since $x\etilde(x)$ has the Cartier property, for almost all prime numbers $p$ and every prime $\kp$ above $p$, we have $\etilde^{(p-1)}+\etilde^p\equiv 0\mod\kp$ and the $p$-curvature of $y'=\etilde y\mod\kp$ vanishes for almost all $\kp$. By Chudnovsky and Chudnovsky's \cite[Theorem 8.1]{CC85}, recalled as Theorem~\ref{theo: CC rank one} in \S\ref{ssec:pcurvature}, all the solutions of $y'=\etilde y$ are algebraic over $\Qbar(x)$.

Let $y$ be a non-zero algebraic solution of $y'=\etilde y$. Since $a_0$ is a rational number, $x^{a_0}y(x/\lambda)$ is a non-zero algebraic solution of $y'=\eta y$, which proves Assertion $\mathrm{(i)}$ and Theorem \ref{theo: main}.
$\hfill\square$ 
\medskip

In this proof, we observe that the vanishing of the $p$-curvature for almost all $p$ implies Gauss's congruences, and consequently congruences modulo higher powers of $p$. This is closely related to the phenomenon described by Movasati in \cite{Movasati}.

\subsection{Proof of Corollary \ref{coro:super}}

Let $\eta$ be an algebraic function over $\Qbar(x)$ and $\delta\in\Qbar$.
The function $(x-\delta)\eta(x)$ has a Puiseux expansion $\sum_{n=r}^\infty p_n (x-\delta)^{n/d}$ at $\delta\in\Qbar$ with $p_n\in\bK$ for all $n$, where $\bK$ is a Galois number field containing $\delta$. We make the change of variables $x=t^d+\delta$ which yields the differential equation $\tilde{y}'(t)=\tilde{\eta}(t)\tilde{y}(t)$, with $\tilde{y}(t):=y(t^d+\delta)$ and 
$$
t\tilde{\eta}(t):=dt^{d}\eta(t^d+\delta)=\sum_{n=r}^\infty dp_nt^{n}\in \bK((t)). 
$$
Applying Theorem~\ref{theo: main} to $\tilde{\eta}(t)$, we obtain that the differential equation $y'=\eta y$ admits a non-zero solution $y$ algebraic over $\Qbar(x)$ if and only if 
$t\tilde{\eta}(t)$ has the Cartier (resp. Gauss) property at $0$. This is equivalent to saying that the sequence $(dp_{n})_{n\in\bZ}$ has the Cartier (resp. Gauss) property, or equivalently that $(x-\delta)\eta(x)$ has the Cartier property at $\delta$. Indeed, if $\kp$ is a prime not dividing $d$, then $(p_n)_{n\in\bZ}$ has Cartier (resp. Gauss) congruences modulo $\kp$ if and only if $(dp_n)_{n\in\bZ}$ has Cartier (resp. Gauss) congruences modulo $\kp$.
\medskip

Similarly, when $\delta=\infty$, the function $x\eta(x)$ has a Puiseux expansion $\sum_{n=r}^\infty p_n x^{-n/d}$ at $\infty$ with $p_n\in\bK$ for all $n$, where $\bK$ is a Galois number field.  The change of variables $x=1/t^d$ yields the differential equation $\tilde{y}'(t)=\tilde{\eta}(t)\tilde{y}(t)$, with $\tilde{y}(t):=y(1/t^d)$ and 
$$
t\tilde{\eta}(t):=-dt^{-d}\eta(t^{-d})=\sum_{n=r}^\infty (-dp_n)t^{n}\in \bK((t)). 
$$
This shows that $y'=\eta y$ has a non-zero algebraic solution over $\Qbar(x)$ if and only if $(-dp_{n})_{n\in\bZ}$ has the Cartier (resp. Gauss) property, which is equivalent to saying that $x\eta(x)$ has the Cartier (resp. Gauss) property at $\infty$.
$\hfill\square$

\bigskip

\noindent {\bf Acknowledgements.} We warmly thank Sara Checcoli, Vasily Golyshev, Maxim Kontsevich and Julien Roques for very useful discussions about this project. Theorem~\ref{theo: HypAvgRoots} was motivated in part by a hypergeometric example considered by Vasily Golyshev: we thank him very much for letting us know his example (unfortunately not covered by our result).

\bigskip

\noindent \'Eric Delaygue, Institut Camille Jordan, 
Universit\'e Claude Bernard Lyon 1, 
43 boulevard du 11 novembre 1918, 
69622 Villeurbanne cedex, France\\
delaygue (at) math.univ-lyon1.fr

\medskip

\noindent Tanguy Rivoal, Institut Fourier, Universit\'e Grenoble Alpes, CNRS, CS 40700, 38058 Grenoble cedex 9, France\\
tanguy.rivoal (at) univ-grenoble-alpes.fr

\bigskip

\noindent Keywords: Abel's problem, Algebraic functions, Cartier and Gauss congruences.

\medskip

\noindent MSC2020: 11A07, 33C20; 34A05, 05A15.


\begin{thebibliography}{1}

\bibitem{AS03} J.-P. Allouche, J. Shallit, {\em Automatic sequences: Theory, Applications, Generalizations}. Cambridge University Press, Cambridge, 2003.

\bibitem{BD79} F. Baldassarri, B. Dwork,  {On second order linear differential equations with algebraic solutions}, {\em Am. J. Math.} \textbf{101} (1979), 42--76.

\bibitem{BGNS23} J. P. Bell, K. Gunn, K. D. Nguyen, and J. C. Saunders, {A general criterion for the P\'olya--Carlson dichotomy and application}, {\em Trans. Am. Math. Soc.} \textbf{376}.6 (2023), 4361--4382.

\bibitem{BGW21} J. Byszewski, G. Graff, and T. Ward, {Dold sequences, periodic points, and dynamics}, {\em Bull. Lond. Math. Soc.} \textbf{53}.5 (2021), 1263--1298.

\bibitem{BH89} F. Beukers, G. Heckman, {Monodromy for the hypergeometric function ${}_nF_{n-1}$}, {\em Invent. Math.} {\bf 95}.2 (1989), 325--354.

\bibitem{BHS18} F. Beukers, M. Houben, A. Straub, {Gauss congruences for rational functions in several variables}, { \em Acta Arith.} {\bf 184}.4 (2018), 341--362.

\bibitem{BV21} F. Beukers, M. Vlasenko, { Dwork crystals I}, { \em Int. Math. Res. Not.}, {\bf 2021}.12 (2021), 8807--8844.

\bibitem{Bos24} A. Bostan, {An arithmetic characterization of some algebraic functions and a new proof of an algebraicity prediction by Golyshev}, preprint, arXiv:2408.13951[math.NT] (2024), to appear in {\em Int. Math. Res. Not.}

\bibitem{BCR24} A. Bostan, X. Caruso, J. Roques, {Algebraic solutions of linear differential equations: an arithmetic approach}, {\em Bull. Am. Math. Soc., New Ser.} \textbf{61}.4 (2024), 609--658.

\bibitem{BSY23} A. Bostan, A. Straub, S. Yurkevich, {On the representability of sequences as constant terms}, {\em J. Number Theory} \textbf{253} (2023), 235--256.

\bibitem{Bou98} A. Boulanger, {Contribution \`a l'\'etude des \'equations diff\'erentielles lin\'eaires homog\`enes int\'egrables alg\'ebriquement}, {\em J. de l'\'Ec. Pol.} (2) \textbf{4} (1898), 1--122.

\bibitem{CC85} D. V. Chudnovsky, G. V. Chudnovsky, {\em Applications of Pad\'e approximations to the Grothendieck conjecture on linear differential equations}, Number theory, Semin. New York 1983-84, Lect. Notes Math. {\bf 1135} (1985), 52--100.

\bibitem{delaygue1} \'E. Delaygue, {Crit\`ere pour l'int\'egralit\'e des coefficients de Taylor des applications miroir}, {\em J. reine angew. Math.} {\bf 662} (2012), 205--252.

\bibitem{DR22}\'E. Delaygue, T. Rivoal, {On Abel's problem and Gauss congruences}, {\em Int. Math. Res. Not.} {\bf 2024}.5 (2024), 4301--4327.

\bibitem{Dwo58} B. Dwork, {Norm residue symbol in local number fields}, {\em  Abh. Math. Sem. Univ. Hamburg} {\bf 22} (1958), 180--190. 

\bibitem{Dwo60} B. Dwork, {On the rationality of the zeta function of an algebraic variety}, {\em Amer. J. Math.} \textbf{82} (1960), 631--648.

\bibitem{DGS94} B. Dwork, G. Gerotto, F. J. Sullivan, {\em An introduction to $G$-functions}, Ann. of Math. Stud. \textbf{133}, Princeton University Press, Princeton, NJ, xxii+323pp, 1994.

\bibitem{Furstenberg} H. Furstenberg, {Algebraic functions over finite fields}, {\em J. Algebra} {\bf 7} (1967), 271--277.

\bibitem{Janichen} W. J\"anichen, {\"Uber die Verallgemeinerung einer Gaussschen Formel aus der Theorie der h\"oheren Kongruenzen}, {\em Sitzungsber. Berl. Math. Ges.} \textbf{20} (1921), 23--29.

\bibitem{Jac37} N. Jacobson, {Abstract derivation and Lie algebras}, {\em Trans. Amer. Math. Soc.} \textbf{42}.2 (1937), 206--224.

\bibitem{KR14} C. Kassel and C. Reutenauer, {Algebraicity of the zeta function associated to a matrix over a free group algebra}, {\em Algebra \& Number Theory} \textbf{8}.2 (2014), 497--511. 

\bibitem{Kon11} M. Kontsevich, {Noncommutative identities}, preprint, arXiv:1109.2469[math.RA] (2011). 

\bibitem{LangANT} S. Lang, {\em Algebraic number theory}, second edition, Graduade Texts in Mathematics, \textbf{110}, Springer-Verlag, New York,  xiv+357pp, 1994.

\bibitem{LY98} B. H. Lian, S.-T. Yau, {\em Integrality of certain exponential series}, Algebra and geometry (Taipei, 1995), 215--227. Lect. Algebra Geom., \textbf{2}, International Press, Cambridge, MA, 1998.

\bibitem{MV16} A. Mellit, M. Vlasenko, {Dwork's congruences for the constant terms of a Laurent polynomial}, {\em Int. J. Number Theory} \textbf{12}.2 (2016), 313--321.

\bibitem{Movasati} H. Movasati, {Grothendieck--Katz conjecture}, preprint, arXiv:2403.11829[math.NT] (2024).

\bibitem{Ris69} R. H. Risch, {The problem of integration in finite terms}, {\em Trans. Am. Math. Soc.} \textbf{139} (1969), 167--189.

\bibitem{SvS15} K. Samol, D. van Straten, {Dwork congruences and reflexive polytopes}, {\em Ann. Math. Qu\'e.} \textbf{39}.2 (2015), 185--203.

\bibitem{serre} J.-P. Serre, {\em Local Fields}, {Graduate Texts in Mathematics} {\bf 67}, Springer-Verlag New York Berlin Heidelberg, 1979.  

\bibitem{Sin80} M. F. Singer, {\em Algebraic solutions of $n$-th order linear differential equations}, Proc. Queen's Number Theory Conf. 1979, Queen's Pap. Pure Appl. Math. \textbf{54} (1980), 379--420.

\bibitem{slater} L. J. Slater, {\em Generalized Hypergeometric Functions}, Cambridge, Cambridge Univ. Press, second edition, 2008.

\bibitem{Steinlein} H. Steinlein, {Fermat's little theorem and Gauss congruence: matrix versions and cyclic permutations}, {\em Am. Math. Mon.} \textbf{124}.6 (2017), 548--553.

\bibitem{Zag18} D. Zagier, {\em The arithmetic and topology of differential equations}, in Proceedings of the European Congress of Mathematics, Berlin, 18-22 July, 2016, Mehrmann, V.; Skutella, M. (Eds.), European Mathematical Society (2018), 717--776.

\bibitem{Zarelua} A. V. Zarelua, {On congruences for the traces of powers of some matrices}, {\em Proc. Steklov Inst. Math.} \textbf{263}.1 (2008), 78--98.

\end{thebibliography}
\end{document}